\documentclass[twocolumn]{my-svjour3}    

\smartqed  

\usepackage{graphicx}
\usepackage{lineno,hyperref}

\usepackage{amsmath,amssymb,mathtools}
\newcommand\supint{
\mathrlap{\phantom{\intop}\mkern-8mu
\overline{\vphantom{\intop}\mkern8mu}}
\mkern-3mu\int}

\newcommand\infint{
\mathrlap{\underline{\vphantom{\intop}\mkern8mu}}
\mkern-3mu\int}
\modulolinenumbers[5]

\begin{document}

\title{Some inequalities for interval-valued functions on time scales}

\author{Dafang Zhao$^{1,2}$  
\and Guoju Ye$^{1}$ 
\and  Wei Liu$^{1}$ 
\and Delfim F. M. Torres$^{3}$}

\authorrunning{D. F. Zhao et al.} 

\institute{Dafang Zhao \at
\email{dafangzhao@163.com}
\and
Guoju Ye \at
\email{yegj@hhu.edu.cn}
\and Wei Liu \at
\email{liuw626@hhu.edu.cn}
\and Delfim F. M. Torres \at
\email{delfim@ua.pt}\\
\and
$^{1}$College of Science,
Hohai University, Nanjing, Jiangsu 210098, P. R. China.\\
$^{2}$School of Mathematics and Statistics, Hubei Normal University, 
Huangshi, Hubei 435002, P. R. China.\\
$^{3}$Center for Research and Development
in Mathematics and Applications (CIDMA),\\
Department of Mathematics, University of Aveiro,
3810-193 Aveiro, Portugal.\\}

% --------------------------------------------

\date{Submitted: 20-Dec-2017  / Revised: 31-Aug-2018 / Accepted: 11-Sept-2018}

\maketitle

% --------------------------------------------

\begin{abstract}
We introduce the interval Darboux delta integral
(shortly, the $ID$ $\Delta$-integral) and the
interval Riemann delta integral (shortly, the $IR$ $\Delta$-integral)
for interval-valued functions on time scales. Fundamental
properties of $ID$ and $IR$ $\Delta$-integrals
and examples are given. Finally, we prove Jensen's,
H\"{o}lder's and Minkowski's inequalities for the $IR$ $\Delta$-integral. 
Also, some examples are given to illustrate our theorems.

\end{abstract}

% --------------------------------------------

\section{Introduction}
\label{intro}

Interval analysis was initiated by Moore for providing reliable computations \cite{M66}.
Since then, interval analysis and interval-valued functions have been extensively studied
both in mathematics and its applications: see, e.g.,
\cite{B13,C13,C15,CB,CC,G17,J01,L15,M79,M09,O15,S09,WG00,Z17}.
Rece-\\ntly, several classical integral inequalities have been extended
to the context of interval-valued functions by  Chalco-Cano et al. \cite{C12,CL15},
Costa \cite{C17}, Costa and Rom\'{a}n-Flores \cite{CF}, Flores-Franuli\v{c} et al. \cite{FC},
Rom\'{a}n-Flores et al. \cite{R16,R13}.

Motivated by \cite{C81,C17,R16}, we introduce the $ID$ and $IR$ $\Delta$-integrals,
and present some integral inequalities on time scales.
A time scale $\mathbb{T}$ is an arbitrary nonempty
closed subset of the real numbers $\mathbb{R}$ with the subspace topology
inherited from the standard topology of $\mathbb{R}$. The theory
of time scales was born in 1988 with the Ph.D. thesis of Hilger \cite{H4}.
The aim is to unify various definitions and results from the theories 
of discrete and continuous dynamical systems,
and to extend them to more general classes of dynamical systems.
It has undergone tremendous expansion and development on various aspects by several
authors over the past three decades: see, e.g.,
\cite{BCT1,BCT2,BP1,BP2,FB15,FT16,V16,W05,YZ,YZT,ZL16}.

In 2013, Lupulescu introduced the Riemann $\Delta$-inte-\\gral 
for interval-valued functions on time scales
and presented some of its basic properties \cite{L13}. 
Nonetheless, to our best knowledge,
there is no systematic theory of integration for interval-valued
functions on time scales. In this work, in order to complete the theory
of $IR$ $\Delta$-integration and improve recent results given in \cite{C81,C17,R16},
we introduce the $ID$ $\Delta$-integral and the $IR$ $\Delta$-integral on time scales.
We show that the $ID$ $\Delta$-integral ia a generalization of the $IR$ $\Delta$-integral.
Also, some basic properties for the $ID$ and $IR$ $\Delta$-integrals, and some examples,
are given. Finally, we present Jensen's inequality, H\"{o}lder's inequality
and Minkowski's inequality for the $IR$ $\Delta$-integral. Some celebrated inequalities
are derived as consequences of our results.

The paper is organized as follows. After a Section~\ref{sec:2} of preliminaries,
in Section~\ref{sec:3} the $ID$ and $IR$ $\Delta$-integrals for interval-valued
functions are introduced. Moreover, some basic properties and examples are given.
In Section~\ref{sec:4}, we prove Jensen's, H\"{o}lder's and Minkowski's inequalities
for the general $IR$ $\Delta$-integral. We end with Section~\ref{sec:5}
of conclusions.

% --------------------------------------------

\section{Preliminaries}
\label{sec:2}

In this section, we recall some basic definitions, notations, properties
and results on interval analysis and the time scale calculus, which are used throughout the
paper. A real interval $[u]$ is the bounded, closed subset of $\mathbb{R}$ defined by
$$
[u]=[\underline{u},\overline{u}]=\{x\in\mathbb{R}|\ \underline{u}\leq x\leq\overline{u}\},
$$
where $\underline{u}, \overline{u}\in \mathbb{R}$ and $\underline{u}\leq\overline{u}$.
The numbers $\underline{u}$ and $\overline{u}$ are called the left and the right endpoints 
of $[\underline{u},\overline{u}]$, respectively. When $\underline{u}$ and $\overline{u}$ 
are equal, the interval $[u]$ is said to be degenerate.
In this paper, the term interval will mean a nonempty interval.
We call $[u]$ positive if $\underline{u}>0$ or negative if $\overline{u}<0$. 
The partial order ``$\leq$" is defined by
$$
[\underline{u},\overline{u}]\leq[\underline{v},\overline{v}]\Longleftrightarrow \underline{u}\leq\underline{v},\overline{u}\leq\overline{v}.
$$
The inclusion ``$\subseteq$" is defined by
$$
[\underline{u},\overline{u}]\subseteq[\underline{v},\overline{v}]\Longleftrightarrow \underline{v}\leq\underline{u},\overline{u}\leq\overline{v}.
$$
For an arbitrary real number $\lambda$ and $[u]$, the interval $\lambda [u]$ is given by
\begin{equation*} \lambda[\underline{u},\overline{u}]=
\begin{cases}
[\lambda\underline{u},\lambda\overline{u}]& \text{if $\lambda>0$},\\
\{0\}& \text{if $\lambda=0$},\\
[\lambda\overline{u},\lambda\underline{u}]& \text{if $\lambda<0$}.
\end{cases}
\end{equation*}
For $[u]=[\underline{u},\overline{u}]$ and $[v]=[\underline{v},\overline{v}]$,
the four arithmetic operators (+,-,$\cdot$,/) are defined by
$$
[u]+[v]=[\underline{u}+\underline{v},\overline{u}+\overline{v}],
$$
$$
[u]-[v]=[\underline{u}-\overline{v},\overline{u}-\underline{v}],
$$
$$
[u]\cdot[v]=\big[\min\{\underline{u}\underline{v},\underline{u}\overline{v},
\overline{u}\underline{v},\overline{u}\overline{v}\},
\max\{\underline{u}\underline{v},\underline{u}\overline{v},
\overline{u}\underline{v},\overline{u}\overline{v}\}\big],
$$
\begin{equation*}
\begin{split}
[u]/[v]=\big[&\min\{\underline{u}/\underline{v},\underline{u}/\overline{v},
\overline{u}/\underline{v},\overline{u}/\overline{v}\},\\
&\max\{\underline{u}/\underline{v},\underline{u}/\overline{v},
\overline{u}/\underline{v},\overline{u}/\overline{v}\}\big], 
{\rm where}\ \  0\notin[\underline{v},\overline{v}].
\end{split}
\end{equation*}
We denote by $\mathbb{R}_{\mathcal{I}}$ the set of all intervals 
of $\mathbb{R}$, and by $\mathbb{R}^{+}_{\mathcal{I}}$ and 
$\mathbb{R}^{-}_{\mathcal{I}}$ the set of all positive intervals 
and negative intervals of $\mathbb{R}$, respectively. 
The Hausdorff--Pompeiu distance between intervals 
$[\underline{u},\overline{u}]$ and $[\underline{v},\overline{v}]$ is defined by 
$$
d\big([\underline{u},\overline{u}],[\underline{v},\overline{v}]\big)
=\max\Big\{|\underline{u}-\underline{v}|,|\overline{u}-\overline{v}|\Big\}.
$$
It is well known that $(\mathbb{R}_{\mathcal{I}}, d)$ is a complete metric space.

Let $\mathbb{T}$ be a time scale. We define the half-open 
interval $[a,b)_{\mathbb{T}}$ by
$$
[a,b)_{\mathbb{T}}
=\left\{t\in \mathbb{T}: a \leq t < b\right\}.
$$
The open and closed intervals are defined similarly. For $t\in \mathbb{T}$,
we denote by $\sigma$ the forward jump operator, i.e., $\sigma(t):=\inf\{s>t: s\in \mathbb{T}\}$,
and by $\rho$ the backward jump operator, i.e., $\rho(t):=\sup\{s<t: s\in \mathbb{T}\}$.
Here, we put $\sigma(\sup\mathbb{T})=\sup\mathbb{T}$ and $\rho(\inf\mathbb{T})=\inf\mathbb{T}$,
where $\sup\mathbb{T}$ and $\inf\mathbb{T}$ are finite. In this situation,
$\mathbb{T}^{\kappa}:=\mathbb{T}\backslash \{\sup\mathbb{T}\}$
and $\mathbb{T}_{\kappa}:=\mathbb{T}\backslash\{\inf\mathbb{T}\}$,
otherwise, $\mathbb{T}^{\kappa}:=\mathbb{T}$ and $\mathbb{T}_{\kappa}:=\mathbb{T}$.
If $\sigma(t)>t$, then we say that $t$ is right-scattered,
while if $\rho(t)<t$, then we say that $t$ is left-scattered.
If $\sigma(t)=t$ and $t< \sup\mathbb{T}$, then $t$ is called right-dense,
and if $\rho(t)=t$ and $t> \inf\mathbb{T}$, then $t$ is left-dense.
The graininess functions $\mu$ and $\eta$ are defined by $\mu(t):=\sigma(t)-t$
and $\eta(t):=t-\rho(t)$, respectively.

A function $f:[a,b]_{\mathbb{T}}\rightarrow \mathbb{R}$ is called 
right-dense continuous ($rd$-continuous) if it is right continuous at each right-dense
point and there exists a finite left limit at all left-dense points. The set of 
$rd$-continuous function $f:[a,b]_{\mathbb{T}}\rightarrow \mathbb{R}$ 
is denoted by $C_{rd}([a,b]_{\mathbb{T}},\mathbb{R})$.

A function $f$ is said to be an interval function of $t$ on $[a,b]_{\mathbb{T}}$ 
if it assigns a nonempty interval 
$$
f(t)=\big[\underline{f}(t), \overline{f}(t)\big]
$$
to each $t\in[a,b]_{\mathbb{T}}$. We say that 
$f:[a,b]_{\mathbb{T}}\rightarrow \mathbb{R}_{\mathcal{I}}$ 
is continuous at $t_{0}\in[a,b]_{\mathbb{T}}$ if for each 
$\epsilon>0$ there exists a $\delta>0$ such that
$$
d(f(t),f(t_{0}))<\epsilon
$$
whenever $|t-t_{0}|<\delta$. The set of continuous function 
$f:[a,b]_{\mathbb{T}}\rightarrow \mathbb{R}_{\mathcal{I}}$ 
is denoted by $C([a,b]_{\mathbb{T}},\mathbb{R}_{\mathcal{I}})$. 
It is clear that $f$ is continuous at $t_{0}$ if and only if 
$\underline{f}$ and $\overline{f}$ are continuous at $t_{0}$.

A division of $[a,b]_{\mathbb{T}}$ is any finite ordered subset 
$D$ having the form
$$
\mathcal{D}=\{a=t_{0}<t_{1}<\cdots <t_{n}=b\}.
$$  
We denote the set of all divisions of $[a,b]_{\mathbb{T}}$ 
by $\mathcal{D}=\mathcal{D}([a,b]_{\mathbb{T}})$.

\begin{lemma}[Bohner and Peterson \cite{BP2}]
\label{lem1}
For every $\delta>0$ there exists some division 
$D\in\mathcal{D}([a,b]_{\mathbb{T}})$ given by 
$$
a=t_{0}<t_{1}<\cdots <t_{n-1}<t_{n}=b
$$ 
such that for each $i\in\{1,2,\ldots,n\}$ either 
$t_{i}-t_{i-1}\leq\delta$ or 
$$
t_{i}-t_{i-1}>\delta 
\ \ {\rm and}\ \ \rho(t_{i})=t_{i-1}.
$$
\end{lemma}

Let $\mathcal{D}(\delta,[a,b]_{\mathbb{T}})$
be the set of all $D\in\mathcal{D}([a,b]_{\mathbb{T}})$ 
that possess the property indicated in Lemma \ref{lem1}.
In each interval $[t_{i-1},t_{i})_{\mathbb{T}}$, where 
$1\leq i\leq n$, choose an arbitrary point $\xi_{i}$ and form the sum
$$
S(f,D,\delta)=\displaystyle \sum^{n}_{i=1}f(\xi_{i})(t_{i}-t_{i-1}),
$$
where $f:[a,b]_{\mathbb{T}}\rightarrow \mathbb{R}(or \ \mathbb{R}_{\mathcal{I}})$. 
We call $S(f,D,\delta)$ a Riemann $\Delta$-sum of $f$ corresponding 
to $D\in \mathcal{D}(\delta,[a,b]_{\mathbb{T}})$.

\begin{definition}[Bohner and Peterson \cite{BP2}]
\label{defn2.1}
A function $f:[a,b]_{\mathbb{T}}\rightarrow \mathbb{R}$
is called Riemann $\Delta$-integrable on $[a,b]_{\mathbb{T}}$
if there exists an $A\in \mathbb{R}$ such that for each
$\epsilon>0$ there exists a $\delta>0$ for which
$$\big|S(f,\mathcal{D},\delta), A\big|<\epsilon$$
for all $D\in \mathcal{D}(\delta,[a,b]_{\mathbb{T}})$. 
In this case, $A$ is called the Riemann $\Delta$-integral of 
$f$ on $[a,b]_{\mathbb{T}}$ and is denoted by $A=(R)\int_{a}^{b}f(t)\Delta t$ 
or $A=\int_{a}^{b}f(t)\Delta t$. The family of all Riemann $\Delta$-integrable 
functions on $[a,b]_{\mathbb{T}}$ is denoted by $\mathcal{R}_{(\Delta,\ [a,b]_{\mathbb{T}})}$.
\end{definition}

% --------------------------------------------

\section{The Interval Darboux and Riemann delta integrals}
\label{sec:3}

Let $f:[a,b]_{\mathbb{T}}\rightarrow \mathbb{R}_{\mathcal{I}}$ be such 
that $f(t)=\big[\underline{f}(t), \overline{f}(t)\big]$ 
for all $t\in [a,b]_{\mathbb{T}}$. We denote
$$
M=\sup\{\overline{f}(t):t\in [a,b)_{\mathbb{T}}\},\ \ m
=\inf\{\underline{f}(t):t\in [a,b)_{\mathbb{T}}\},
$$
and for $1\leq i\leq n$,
$$
M_{i}=\sup\{\overline{f}(t):t\in [t_{i-1},t_{i})_{\mathbb{T}}\},
$$
$$ 
m_{i}=\inf\{\underline{f}(t):t\in [t_{i-1},t_{i})_{\mathbb{T}}\}.
$$
The lower Darboux $\Delta$-sum $L(\underline{f},D)$ of $\underline{f}$
with respect to $D\in\mathcal{D}([a,b]_{\mathbb{T}})$ is the sum
$$
L(\underline{f},D)=\sum_{i=1}^{n}m_{i}(t_{i}-t_{i-1}),
$$
and the upper Darboux $\Delta$-sum $U(\overline{f},D)$ is
$$
U(\overline{f},D)=\sum_{i=1}^{n}M_{i}(t_{i}-t_{i-1}).
$$

\begin{definition}[The Interval Darboux delta integral]
\label{defn3}
Let $I=[a,b]_{\mathbb{T}}$, where $a,b \in \mathbb{T}$. 
The lower Darboux $\Delta$-integral of $\underline{f}$ 
on $[a,b]_{\mathbb{T}}$ is defined by
$$
(D)\infint_{a}^{b}\underline{f}(t)\Delta t
=\sup_{D\in \mathcal{D}([a,b]_{\mathbb{T}})}\Big\{L(\underline{f},D)\Big\}
$$
and the upper Darboux $\Delta$-integral of $\overline{f}$ 
on $[a,b]_{\mathbb{T}}$ is defined by
$$
(D)\supint_{a}^{b}\overline{f}(t)\Delta t
=\inf_{D\in \mathcal{D}([a,b]_{\mathbb{T}})}\Big\{U(\overline{f},D)\Big\}.
$$
Then, we define the $ID$ $\Delta$-integral of
$f:[a,b]_{\mathbb{T}}\rightarrow \mathbb{R}_{\mathcal{I}}$ 
on $[a,b]_{\mathbb{T}}$ as the interval
$$
(ID)\int_{a}^{b}f(t)\Delta t=\Bigg[(D)\infint_{a}^{b}\underline{f}(t)\Delta t, 
(D)\supint_{a}^{b}\overline{f}(t)\Delta t\Bigg].
$$
The family of all $ID$ $\Delta$-integrable functions 
on $[a,b]_{\mathbb{T}}$ is denoted by 
$\mathcal{ID}_{(\Delta,\ [a,b]_{\mathbb{T}})}$.
\end{definition}

\begin{theorem}
\label{thm2}
Let $f:[a,b]_{\mathbb{T}}\rightarrow \mathbb{R}_{\mathcal{I}}$ 
be such that 
$$
f(t)=\big[\underline{f}(t), \overline{f}(t)\big]
$$
for all $t\in [a,b]_{\mathbb{T}}$. Then 
$f\in \mathcal{ID}_{(\Delta,\ [t,\sigma(t)]_{\mathbb{T}})}$ and
$$
(ID)\int_{t}^{\sigma(t)}f(s)\Delta s=\big[\mu(t)\underline{f}(t),\mu(t)\overline{f}(t)\big].
$$
\end{theorem}

\begin{proof}
If $\sigma(t)=t$, then the result is obvious. If $\sigma(t)>t$, 
then $\mathcal{D}([a,b]_{\mathbb{T}})$ only contains 
one single element given by
$$
t=s_{0}<s_{1}=\sigma(t).
$$
Since $[s_{0},s_{1})=[t,\sigma(t))=\{t\}$, we have
$$
L(f,D)=\underline{f}(t)(\sigma(t)-t)=\mu(t)\underline{f}(t),
$$
$$
U(f,D)=\overline{f}(t)(\sigma(t)-t)=\mu(t)\overline{f}(t).
$$
Consequently, we obtain
$$
(ID)\int_{t}^{\sigma(t)} f(s)\Delta s
=\big[\mu(t)\underline{f}(t),\mu(t)\overline{f}(t)\big].
$$
The result is proved.
\end{proof}

\begin{remark}
\label{rmk3.1}
It is clear that if $f$ is a real-valued function,
then our Definition~\ref{defn3} implies the definition of Darboux $\Delta$-integral
introduced by \cite{BP2}. We also have the following:

\noindent(1) If $\mathbb{T}=\mathbb{R}$, then Definition~\ref{defn3}
implies the definition of Darboux interval integral 
introduced by  Caprani et al. \cite{C81,R82}.

\noindent(2) If $\mathbb{T}=\mathbb{Z}$, then each function 
$f:\mathbb{Z}\rightarrow \mathbb{R}_{\mathcal{I}}$ is 
$ID$ $\Delta$-integrable on $[a,b]_{\mathbb{T}}$. Moreover,
$$
(ID)\int_{a}^{b}f(t)\Delta t=\Bigg[\sum_{t=a}^{b-1} 
\underline{f}(t), \sum_{t=a}^{b-1}\overline{f}(t)\Bigg].
$$

\noindent(3) If $\mathbb{T}=h\mathbb{Z}$, then each function 
$f:h\mathbb{Z}\rightarrow \mathbb{R}_{\mathcal{I}}$ 
is $ID$ $\Delta$-integrable on $[a,b]_{\mathbb{T}}$. Moreover,
$$
(ID)\int_{a}^{b}f(t)\Delta t=\Bigg[
\sum_{k=\frac{a}{h}}^{\frac{b}{h}-1}\underline{f}(kh)h, 
\sum_{k=\frac{a}{h}}^{\frac{b}{h}-1}\overline{f}(kh)h\Bigg].
$$
\end{remark}

\begin{example}
\label{ex1}
Suppose that $[a,b]_{\mathbb{T}}=[0,1]$, $\mathbb{Q}$ is the set 
of rational numbers in $[0,1]$, and 
$f:[a,b]_{\mathbb{T}}\rightarrow \mathbb{R}_{\mathcal{I}}$ is defined by
\begin{equation*}
f(t)=
\begin{cases}
[-1,0],
& \text{if $t\in \mathbb{Q}$},\\
[1,2],
& \text{if $t\in[0,1]\backslash \mathbb{Q}$}.
\end{cases}
\end{equation*}
Then,
\begin{equation*}
\begin{split}
(ID)\int_{0}^{1}f(t)\Delta t&=\Bigg[(D)\infint_{0}^{1}
\underline{f}(t)dt, (D)\supint_{0}^{1}\overline{f}(t)dt\Bigg]\\
&=[-1,2].
\end{split}
\end{equation*}
\end{example}

\begin{example}
\label{ex2}
Suppose that $[a,b]_{\mathbb{T}}=\big\{0,\frac{1}{3},\frac{1}{2},1\big\}$ and \\
$f:[a,b]_{\mathbb{T}}\rightarrow \mathbb{R}_{\mathcal{I}}$ is defined by
\begin{equation*}
f(t)=
\begin{cases}
[-1,0], & \text{if $t=0$},\\
[-\frac{1}{3},\frac{1}{3}], & \text{if $t=\frac{1}{3}$},\\
[-\frac{1}{2},\frac{1}{2}], & \text{if $t=\frac{1}{2}$},\\
[1,2], & \text{if $t=1$}.
\end{cases}
\end{equation*}
Then,
\begin{equation*}
\begin{split}
(D)\infint_{0}^{1}\underline{f}(t)\Delta t
&=(-1)\cdot\frac{1}{3}+\left(-\frac{1}{3}\right)\cdot\frac{1}{6}
+\left(-\frac{1}{2}\right)\cdot\frac{1}{2}\\
&=-\frac{23}{36},
\end{split}
\end{equation*}
\begin{equation*}
(D)\supint_{0}^{1}\overline{f}(t)\Delta t=0\cdot\frac{1}{3}
+\frac{1}{3}\cdot\frac{1}{6}+\frac{1}{2}\cdot\frac{1}{2}=\frac{11}{36},
\end{equation*}
and therefore
$$
(ID)\int_{0}^{1}f(t)\Delta t=\Bigg[-\frac{23}{36}, \frac{11}{36}\Bigg].
$$
\end{example}

\begin{theorem}
\label{thm3}
Let $f,\ g\in \mathcal{ID}_{(\Delta,\ [a,b]_{\mathbb{T}})}$,
and $\lambda$ be \\an arbitrary real number. Then,

\noindent(1) $\lambda f \in \mathcal{ID}_{(\Delta,\ [a,b]_{\mathbb{T}})}$ and
$$
(ID)\int_{a}^{b}\lambda f(t)\Delta t=\lambda (ID)\int_{a}^{b}f(t)\Delta t;
$$

\noindent(2) $f+g \in \mathcal{ID}_{(\Delta,\ [a,b]_{\mathbb{T}})}$ and
\begin{equation*}
\begin{split}
&(ID)\int_{a}^{b}(f(t)+g(t))\Delta t\\
&\subseteq(ID)\int_{a}^{b}f(t)\Delta t+(ID)\int_{a}^{b}g(t)\Delta t;
\end{split}
\end{equation*}

\noindent(3) for $c\in [a,b]_{\mathbb{T}}$ and $a<c<b$,
$$(ID)\int_{a}^{c}f(t)\Delta t
+(ID)\int_{c}^{b} f(t)\Delta t=(ID)\int_{a}^{b} f(t)\Delta t;
$$

\noindent(4) if $f\subseteq g$ on $[a,b]_{\mathbb{T}}$, then
$$
(ID)\int_{a}^{b} f(t)\Delta t\subseteq (ID)\int_{a}^{b} g(t)\Delta t.
$$
\end{theorem}

\begin{proof}
We only prove that part (2) of Theorem~\ref{thm3} holds. 
The other relations are obvious. Suppose that
$$
f(t)=\big[\underline{f}(t),\overline{f}(t)\big], \ g(t)
=\big[\underline{g}(t),\overline{g}(t)\big].
$$
Select any division $D\in\mathcal{D}([a,b]_{\mathbb{T}})$ 
having the form
$$
D=\{a=t_{0}<t_{1}<\cdots <t_{n}=b\}.
$$
Then,
\begin{equation*}
\begin{split}
&\inf_{t\in [t_{i-1},t_{i})_{\mathbb{T}}}\{\underline{f}(t)\}
+\inf_{t\in [t_{i-1},t_{i})_{\mathbb{T}}}\{\underline{g}(t)\}\\
&\leq \inf_{t\in [t_{i-1},t_{i})_{\mathbb{T}}}\{\underline{f}(t)
+\underline{g}(t)\},
\end{split}
\end{equation*}
\begin{equation*}
\begin{split}
&\sup_{t\in [t_{i-1},t_{i})_{\mathbb{T}}}\{\overline{f}(t)+\overline{g}(t)\}\\
&\leq\sup_{t\in [t_{i-1},t_{i})_{\mathbb{T}}}\{\overline{f}(t)\}
+\sup_{t\in [t_{i-1},t_{i})_{\mathbb{T}}}\{\overline{g}(t)\},
 \end{split}
\end{equation*}
and it follows that
$$
L(\underline{f},D)+L(\underline{g},D)\leq L(\underline{f}+\underline{g},D),
$$
$$
U(\underline{f},D)+U(\underline{g},D)\geq U(\underline{f}+\underline{g},D).
$$
The intended result follows.
\end{proof}

\begin{example}
\label{ex3}
Suppose that $[a,b]_{\mathbb{T}}=[0,1]$, $\mathbb{Q}$ is the set 
of rational numbers in $[0,1]$, and $f,g:[a,b]_{\mathbb{T}}
\rightarrow \mathbb{R}_{\mathcal{I}}$ are defined by
\begin{equation*}
f(t)=
\begin{cases}
[-1,0],
& \text{if $t\in \mathbb{Q}$},\\
[1,2],
& \text{if $t\in[0,1]\backslash \mathbb{Q}$},
\end{cases}
\end{equation*}
\begin{equation*}
g(t)=
\begin{cases}
[0,1],
& \text{if $t\in \mathbb{Q}$},\\
[-2,-1],
& \text{if $t\in[0,1]\backslash \mathbb{Q}$}.
\end{cases}
\end{equation*}
Then
$$
f(t)+g(t)=[-1,1]
$$
for all $t\in[0,1]$. It follows that
\begin{equation*}
\begin{split}
&(ID)\int_{0}^{1}f(t)\Delta t+(ID)\int_{0}^{1}g(t)\Delta t\\
&=[-1,2]+[-2,1]\\
&=[-3,3],
\end{split}
\end{equation*}
\begin{equation*}
 (ID)\int_{0}^{1}(f(t)+g(t))\Delta t =[-1,1].
\end{equation*}
Therefore, we have
\begin{equation*}
\begin{split}
&(ID)\int_{a}^{b}(f(t)+g(t))\Delta t\\
&\subseteq (ID)\int_{a}^{b}f(t)\Delta t+(ID)\int_{a}^{b}g(t)\Delta t.
\end{split}
\end{equation*}
\end{example}

We now give Riemann's definition of integrability, which is equivalent to the
Riemann $\Delta$-integral given in \cite[Definition 13]{L13}.

\begin{definition}[The Interval Riemann delta integral]
\label{defn4}
A function $f:[a,b]_{\mathbb{T}}\rightarrow \mathbb{R}_{\mathcal{I}}$ 
is called $IR$ $\Delta$-integrable on $[a,b]_{\mathbb{T}}$ if there exists 
an $A\in \mathbb{R}_{\mathcal{I}}$ such that for each
$\epsilon>0$ there exists a $\delta>0$ for which
$$d\big(S(f,\mathcal{D},\delta), A\big)<\epsilon$$
for all $D\in \mathcal{D}(\delta,[a,b]_{\mathbb{T}})$. In this case, 
$A$ is called the $IR$ $\Delta$-integral of $f$ on $[a,b]_{\mathbb{T}}$ 
and is denoted by $A=(IR)\int_{a}^{b}f(t)\Delta t$. 
The family of all $IR$ $\Delta$-integrable functions 
on $[a,b]_{\mathbb{T}}$ is denoted by $\mathcal{IR}_{(\Delta,\ [a,b]_{\mathbb{T}})}$.
\end{definition}

\begin{remark} 
\label{rmk3.2}
Definitions \ref{defn3} and \ref{defn4} are not equivalent.
If $f\in\mathcal{IR}_{(\Delta,\ [a,b]_{\mathbb{T}})}$, then 
$f\in\mathcal{ID}_{(\Delta,\ [a,b]_{\mathbb{T}})}$. However, 
the converse is not always true (see Example \ref{ex1}). 
It is clear that $f\in\mathcal{ID}_{(\Delta,\ [a,b]_{\mathbb{T}})}$, 
but $f\notin\mathcal{IR}_{(\Delta,\ [a,b]_{\mathbb{T}})}$. In fact, 
all bounded interval functions are $ID$ $\Delta$-integrable, 
but boundedness of $f$ is not a sufficient condition for 
$IR$ $\Delta$-integrability. If $f$ is a continuous function, 
then $f\in\mathcal{IR}_{(\Delta,\ [a,b]_{\mathbb{T}})}$ 
if and only if $f\in\mathcal{ID}_{(\Delta,\ [a,b]_{\mathbb{T}})}$, 
in which case the value of the integrals agree.
\end{remark}

The following two theorems can be easily verified and so the proofs are omitted.

\begin{theorem}
\label{thm4}
If $f\in C([a,b]_{\mathbb{T}},\mathbb{R}_{\mathcal{I}})$, 
then $f\in\mathcal{IR}_{(\Delta,\ [a,b]_{\mathbb{T}})}$ and 
$$
(IR)\int_{a}^{b}f(t)\Delta t=\Bigg[\int_{a}^{b}
\underline{f}(t)\Delta t,\int_{a}^{b}\overline{f}(t)\Delta t\Bigg].
$$
\end{theorem}

\begin{theorem}
\label{thm5}
Let $f,\ g\in \mathcal{IR}_{(\Delta,\ [a,b]_{\mathbb{T}})}$, 
and $\lambda$ be an arbitrary real number. Then,

\noindent(1) $\lambda f \in \mathcal{IR}_{(\Delta,\ [a,b]_{\mathbb{T}})}$ and
$$
(IR)\int_{a}^{b}\lambda f(t)\Delta t=\lambda (IR)\int_{a}^{b}f(t)\Delta t;
$$

\noindent(2) $f+g \in \mathcal{IR}_{(\Delta,\ [a,b]_{\mathbb{T}})}$ and
\begin{equation*}
\begin{split}
&(IR)\int_{a}^{b}(f(t)+g(t))\Delta t\\
&=(IR)\int_{a}^{b}f(t)\Delta t+(IR)\int_{a}^{b}g(t)\Delta t;
\end{split}
\end{equation*}

\noindent(3) for $c\in [a,b]_{\mathbb{T}}$ and $a<c<b$,
$$
(IR)\int_{a}^{c}f(t)\Delta t+(IR)\int_{c}^{b} f(t)\Delta t
=(IR)\int_{a}^{b} f(t)\Delta t;
$$

\noindent(4) if $f\subseteq g$ on $[a,b]_{\mathbb{T}}$, then
$$
(IR)\int_{a}^{b} f(t)\Delta t\subseteq (IR)\int_{a}^{b} g(t)\Delta t.
$$
\end{theorem}

\begin{example}
\label{ex4}
Suppose that $\mathbb{T}=[-1,0]\cup 3^{\mathbb{N}_{0}}$, 
where $[-1,0]$ is a real-valued interval
and $\mathbb{N}_{0}$ is the set of nonnegative integers. Let
$f:[a,b]_{\mathbb{T}}\rightarrow \mathbb{R}_{\mathcal{I}}$ be defined by
\begin{equation*}
f(t)=
\begin{cases}
[t,t+1], & \text{if $t\in [-1,0)$},\\
[1,2], & \text{if $t=0$},\\
[t,t^{2}+1], & \text{if $t\in 3^{\mathbb{N}_{0}}$}.
\end{cases}
\end{equation*}
If $[a,b]_{\mathbb{T}}=[-1,3]_{\mathbb{T}}$, then
\begin{equation*}
\begin{split}
&(IR)\int_{-1}^{3}f(t)\Delta t\\
&=\Bigg[\int_{-1}^{3}\underline{f}(t)\Delta t,\int_{-1}^{3}\overline{f}(t)\Delta t\Bigg]\\
&=\Bigg[\int_{-1}^{0}tdt+(R)\int_{0}^{1}\Delta t+\int_{1}^{3}t\Delta t,\\
&\ \ \ \ \ \ \int_{-1}^{0}(t+1)dt+\int_{0}^{1}2\Delta t+\int_{1}^{3}(t^{2}+1)\Delta t\Bigg]\\
&=\Bigg[\frac{1}{2}t^{2}\Big|_{-1}^{0}+1+2t^{2}\Big|_{1},\\
&\ \ \ \ \ \ \ \frac{1}{2}(t^{2}+t)\Big|_{-1}^{0}+2+2t(t^{2}+1)\Big|_{1}\Bigg]\\
&=\Big[2\frac{1}{2},6\frac{1}{2}\Big].
\end{split}
\end{equation*}
\end{example}

% --------------------------------------------

\section{Some inequalities for the interval Riemann delta integral}
\label{sec:4}

We begin by recalling the notions of convexity on time scales.

\begin{definition}[Dinu \cite{D08}]
\label{defn4.1}
We say that $f:[a,b]_{\mathbb{T}}\rightarrow \mathbb{R}$
is a convex function if for all $x,y\in[a,b]_{\mathbb{T}}$ 
and $\alpha\in[0,1]$ we have
\begin{equation}
\label{1}
f(\alpha x+(1-\alpha)y)\leq \alpha f(x)+(1-\alpha)f(y)
\end{equation}
for which $\alpha x+(1-\alpha)y\in[a,b]_{\mathbb{T}}$.
If inequality \eqref{1} is reversed, then $f$ is said to be concave. 
If $f$ is both convex and concave, then $f$ is said to be affine. 
The set of all convex, concave and affine interval-valued functions 
are denoted by $SX([a,b]_{\mathbb{T}},\mathbb{R})$, 
$SV([a,b]_{\mathbb{T}},\mathbb{R})$, 
and $SA([a,b]_{\mathbb{T}},\mathbb{R})$, respectively.
\end{definition}

We can now introduce the concept of interval-valued convexity.

\begin{definition}
\label{defn4.2}
We say that $f:[a,b]_{\mathbb{T}}\rightarrow \mathbb{R}_{\mathcal{I}}$ 
is a convex interval-valued function if for all 
$x,y\in[a,b]_{\mathbb{T}}$ and $\alpha\in(0,1)$ we have
\begin{equation}
\label{4}
\alpha f(x)+(1-\alpha)f(y) \subseteq f(\alpha x+(1-\alpha)y)
\end{equation}
for which $\alpha x+(1-\alpha)y\in[a,b]_{\mathbb{T}}$.
If the set inclusion \eqref{4} is reversed, then $f$ is said to be concave. 
If $f$ is both convex and concave, then $f$ is said to be affine. 
The set of all convex, concave and affine interval-valued functions 
are denoted by $SX([a,b]_{\mathbb{T}},\mathbb{R}_{\mathcal{I}})$, 
$SV([a,b]_{\mathbb{T}},\mathbb{R}_{\mathcal{I}})$ 
and $SA([a,b]_{\mathbb{T}}, \mathbb{R}_{\mathcal{I}})$, respectively.
\end{definition}

\begin{remark}
\label{rmk4.0}
It is clear that if $\mathbb{T}=\mathbb{R}$, then Definition~\ref{defn4.2}
implies the definition of convexity introduced by Breckner \cite{B}.
\end{remark}

\begin{theorem}
\label{thm4.1}
Let $f:[a,b]_{\mathbb{T}}\rightarrow \mathbb{R}_{\mathcal{I}}$ 
be such that 
$$
f(t)=[\underline{f}(t),\overline{f}(t)]
$$ 
for all $t\in [a,b]_{\mathbb{T}}$. Then,

\noindent(1) $f\in SX([a,b]_{\mathbb{T}},\mathbb{R}_{\mathcal{I}})$ 
if and only if $\underline{f}\in SX([a,b]_{\mathbb{T}},\mathbb{R})$ and 
$\overline{f}\in SV([a,b]_{\mathbb{T}},\mathbb{R})$,

\noindent(2) $f\in SV([a,b]_{\mathbb{T}},\mathbb{R}_{\mathcal{I}})$
if and only if $\underline{f}\in SV([a,b]_{\mathbb{T}},\mathbb{R})$
and $\overline{f}\in SX([a,b]_{\mathbb{T}},\mathbb{R})$,

\noindent(3) $f\in SA([a,b]_{\mathbb{T}},\mathbb{R}_{\mathcal{I}})$
if and only if $\underline{f}, \overline{f}
\in SA([a,b]_{\mathbb{T}},\mathbb{R})$.
\end{theorem}

\begin{proof}
We only prove that part (1) of Theorem~\ref{thm4.1} holds.
Suppose that $f\in SX([a,b]_{\mathbb{T}},\mathbb{R}_{\mathcal{I}})$
and consider $x,y\in[a,b]_{\mathbb{T}}$, $\alpha\in[0,1]$. Then,
$$
\alpha f(x)+(1-\alpha)f(y) \subseteq f(\alpha x+(1-\alpha)y),
$$
that is,
\begin{equation}
\label{5}
\begin{split}
&\big[\alpha\underline{f}(x)+(1-\alpha)\underline{f}(y),
\alpha\overline{f}(x)+(1-\alpha)\overline{f}(y)\big]\\ 
&\subseteq \big[\underline{f}(\alpha x+(1-\alpha)y),
\overline{f}(\alpha x+(1-\alpha)y)\big].
\end{split}
\end{equation}
It follows that
$$
\alpha\underline{f}(x)+(1-\alpha)\underline{f}(y)
\geq \underline{f}(\alpha x+(1-\alpha)y)
$$
and
$$
\alpha\overline{f}(x)+(1-\alpha)\overline{f}(y)
\leq \overline{f}(\alpha x+(1-\alpha)y).
$$
This shows that
$$
\underline{f}\in SX([a,b]_{\mathbb{T}},\mathbb{R})\  
{\rm and }\ \overline{f}\in SV([a,b]_{\mathbb{T}},\mathbb{R}).
$$
Conversely, if
$$
\underline{f}\in SX([a,b]_{\mathbb{T}},\mathbb{R})\  
{\rm and }\ \overline{f}\in SV([a,b]_{\mathbb{T}},\mathbb{R}),
$$
by Definition~\ref{defn4.2} and the set inclusion \eqref{5},
we have \\$f\in SX([a,b]_{\mathbb{T}},\mathbb{R}_{\mathcal{I}})$.
\end{proof}

\begin{theorem}[Dinu \cite{D08}]
\label{thm4.2}
A convex function on \\$[a,b]_{\mathbb{T}}$ 
is continuous on $(a,b)_{\mathbb{T}}$.
\end{theorem}

\begin{theorem}
\label{thm4.4}
Let $f:[a,b]_{\mathbb{T}}\rightarrow \mathbb{R}_{\mathcal{I}}$ 
be such that 
$$
f(t)=[\underline{f}(t),\overline{f}(t)]
$$
for all $t\in [a,b]_{\mathbb{T}}$. If 
$$
f\in SX([a,b]_{\mathbb{T}},\mathbb{R}_{\mathcal{I}})
\cup SV([a,b]_{\mathbb{T}},\mathbb{R}_{\mathcal{I}})
\cup SA([a,b]_{\mathbb{T}},\mathbb{R}_{\mathcal{I}}),
$$ 
then $f\in\mathcal{IR}_{(\Delta,\ [a,b]_{\mathbb{T}})}$.
\end{theorem}

\begin{proof}
Suppose that
$$
f\in SV([a,b]_{\mathbb{T}},\mathbb{R}_{\mathcal{I}})
\cup SV([a,b]_{\mathbb{T}},\mathbb{R}_{\mathcal{I}})
\cup SA([a,b]_{\mathbb{T}},\mathbb{R}_{\mathcal{I}}).
$$
Due to Theorems~\ref{thm4.1} and \ref{thm4.2},
it follows that $\underline{f}$ and $\overline{f}$ 
are continuous. Then, from Theorem~5.19 of \cite{BP2}, we have that
$$
\overline{f}(t), \underline{f}(t)\in\mathcal{R}_{(\Delta,\ [a,b]_{\mathbb{T}})}.
$$
Hence, $f\in\mathcal{IR}_{(\Delta,\ [a,b]_{\mathbb{T}})}$.
\end{proof}

\begin{theorem}[Wong et al. \cite{W06}]
\label{thm4.5}
Let $a,b\in[a,b]_{\mathbb{T}}$ and $c,d\in\mathbb{R}$. Suppose that
$g\in C_{rd}([a,b]_{\mathbb{T}},(c,d))$ and 
$h\in C_{rd}([a,b]_{\mathbb{T}},\mathbb{R})$ with
$$
\int_{a}^{b}|h(s)|\Delta s>0.
$$
If $f\in C((c,d),\mathbb{R})$ is convex, then
\begin{equation}\label{6}
\begin{split}
&f\Bigg(\frac{\int_{a}^{b}|h(s)|g(s)\Delta s}{\int_{a}^{b}|h(s)|\Delta s}\Bigg)
\leq \frac{\int_{a}^{b}|h(s)|f(g(s))\Delta s}{\int_{a}^{b}|h(s)|\Delta s}.
\end{split}
\end{equation}
If $f$ is concave, then inequality \eqref{6} is reversed.
\end{theorem}

\begin{theorem}[Jensen's inequality]
\label{thm4.6}
\\Let $g\in C_{rd}([a,b]_{\mathbb{T}},(c,d))$ 
and $h\in C_{rd}([a,b]_{\mathbb{T}},\mathbb{R})$ with
$$
\int_{a}^{b}|h(s)|\Delta s>0.
$$
If $f\in C((c,d),\mathbb{R}^{+}_{\mathcal{I}})$ 
is a convex function, then
\begin{equation*}
\begin{split}
&\frac{(IR)\int_{a}^{b}|h(s)|f(g(s))\Delta s}{\int_{a}^{b}|h(s)|\Delta s}
\subseteq f\Bigg(\frac{\int_{a}^{b}|h(s)|g(s)\Delta s}{\int_{a}^{b}|h(s)|\Delta s}\Bigg).
\end{split}
\end{equation*}
\end{theorem}

\begin{proof}
By hypothesis, we have 
$$
|h|\overline{f(g)},\ \  
|h|\underline{f(g)}\in\mathcal{R}_{(\Delta,\ [a,b])}.
$$ 
Hence, $|h|f(g)\in\mathcal{IR}_{(\Delta,\ [a,b])}$ and
\begin{equation*}
\begin{split}
&(IR)\int_{a}^{b}|h(s)|f(g(s))\Delta s\\
&=\Bigg[\int_{a}^{b}|h(s)|\underline{f(g)}(s)\Delta s,\int_{a}^{b}|h(s)|\overline{f(g)}(s)\Delta s\Bigg].
\end{split}
\end{equation*}
From Theorem~\ref{thm4.5}, it follows that
$$
\underline{f}\Bigg(\frac{\int_{a}^{b}|h(s)|g(s)\Delta s}{(\int_{a}^{b}|h(s)|\Delta s}\Bigg)\leq\frac{\int_{a}^{b}|h(s)|\underline{f(g)}(s)\Delta s}{\int_{a}^{b}|h(s)|\Delta s}
$$
and
$$
\overline{f}\Bigg(\frac{\int_{a}^{b}|h(s)|g(s)\Delta s}{\int_{a}^{b}|h(s)|\Delta s}\Bigg)\geq\frac{\int_{a}^{b}|h(s)|\overline{f(g)}(s)\Delta s}{\int_{a}^{b}|h(s)|\Delta s},
$$
which implies
\begin{equation*}
\begin{split}
&\Bigg[\frac{\int_{a}^{b}|h(s)|\underline{f(g)}(s)\Delta s}{\int_{a}^{b}|h(s)|\Delta s}, \frac{\int_{a}^{b}|h(s)|\overline{f(g)}(s)\Delta s}{\int_{a}^{b}|h(s)|\Delta s}\Bigg]\\
&\ \ \subseteq \Bigg[\underline{f}\Bigg(\frac{\int_{a}^{b}|h(s)|g(s)\Delta s}{\int_{a}^{b}|h(s)|\Delta s}\Bigg),\overline{f}\Bigg(\frac{\int_{a}^{b}|h(s)|g(s)\Delta s}{\int_{a}^{b}|h(s)|\Delta s}\Bigg)\Bigg],
\end{split}
\end{equation*}
that is,
\begin{equation*}
\begin{split}
&\frac{\Bigg[\int_{a}^{b}|h(s)|\underline{f(g)}(s)\Delta s,
\int_{a}^{b}|h(s)|\overline{f(g)}(s)\Delta s\Bigg]}{\int_{a}^{b}|h(s)|\Delta s}\\
&\ \ \subseteq \Bigg[\underline{f}\Bigg(\frac{\int_{a}^{b}|h(s)|g(s)\Delta s}{\int_{a}^{b}|h(s)|\Delta s}\Bigg),\overline{f}\Bigg(\frac{\int_{a}^{b}|h(s)|g(s)\Delta s}{\int_{a}^{b}|h(s)|\Delta s}\Bigg)\Bigg].
\end{split}
\end{equation*}
Finally, we obtain
\begin{equation*}
\frac{(IR)\int_{a}^{b}|h(s)|f(g(s))\Delta s}{\int_{a}^{b}|h(s)|\Delta s}
\subseteq f\Bigg(\frac{\int_{a}^{b}|h(s)|g(s)\Delta s}{\int_{a}^{b}|h(s)|\Delta s}\Bigg).
\end{equation*}
The proof is complete.
\end{proof}

\begin{example}
\label{ex5}
Suppose that $[a,b]_{\mathbb{T}}=[0,1]\cup \{\frac{3}{2}\}$, 
where $[0,1]$ is a real-valued interval. Let
$g(s)=s^{2}$, $h(s)=e^{s}$, and
$f(s)=[s^{2},4\sqrt{s}]$.
Then
\begin{equation*}
\begin{split}
&\frac{(IR)\int_{a}^{b}|h(s)|f(g(s))\Delta s}{\int_{a}^{b}|h(s)|\Delta s}\\
&=\frac{(IR)\int_{0}^{\frac{3}{2}}\big[s^{4}e^{s},4se^{s}\big]
\Delta s}{\int_{0}^{\frac{3}{2}}e^{s}\Delta s}\\
&=\frac{\bigg[\int_{0}^{\frac{3}{2}}s^{4}e^{s}\Delta s,
\int_{0}^{\frac{3}{2}}4se^{s}\Delta s\bigg]}{\int_{0}^{\frac{3}{2}}e^{s}\Delta s}\\
&=\frac{\bigg[\int_{0}^{1}s^{4}e^{s}ds+\int_{1}^{\frac{3}{2}}s^{4}e^{s}\Delta s,\int_{0}^{1}4se^{s}ds+\int_{1}^{\frac{3}{2}}4se^{s}\Delta s\bigg]}{
\int_{0}^{1}e^{s}ds+\int_{1}^{\frac{3}{2}}e^{s}\Delta s}\\
&=\frac{\bigg[9\frac{1}{2}e-24,4+2e\bigg]}{\frac{3}{2}e-1}\\
&=\Bigg[\frac{19e-48}{3e-2},\frac{8+4e}{3e-2}\Bigg],
\end{split}
\end{equation*}
and
\begin{equation*}
\begin{split}
&f\Bigg(\frac{\int_{a}^{b}|h(s)|g(s)\Delta s}{\int_{a}^{b}|h(s)|\Delta s}\Bigg)\\
&=f\Bigg(\frac{\int_{0}^{\frac{3}{2}}s^{2}e^{s}\Delta s}{\int_{0}^{\frac{3}{2}}e^{s}\Delta s}\Bigg)\\
&=f\Bigg(\frac{\int_{0}^{1}s^{2}e^{s}ds+\int_{1}^{\frac{3}{2}}s^{2}e^{s}\Delta s}{\frac{3}{2}e-1}\Bigg)\\
&=f\Bigg(\frac{\frac{3}{2}e-2}{\frac{3}{2}e-1}\Bigg)\\
&=\Bigg[\bigg(\frac{3e-4}{3e-2}\bigg)^{2},4\sqrt{\frac{3e-4}{3e-2}}\Bigg].
\end{split}
\end{equation*}
It follows that
\begin{equation*}
\begin{split}
&\Bigg[\frac{19e-48}{3e-2},\frac{8+4e}{3e-2}\Bigg]
\subseteq \Bigg[\bigg(\frac{3e-4}{3e-2}\bigg)^{2},4\sqrt{\frac{3e-4}{3e-2}}\Bigg].
\end{split}
\end{equation*}
\end{example}

It is clear that if $[a,b]_{\mathbb{T}}=[0,1]$ and $h(s)\equiv1$,
then we get a similar result given in \cite[Theorem~3.5]{C17} 
by T. M. Costa. Similarly, we can get the following results
that generalize \cite[Theorem~3.4]{C17} and \cite[Corollary 3.3]{C17}.

\begin{theorem}
\label{thm4.7}
Let $g\in C_{rd}([a,b]_{\mathbb{T}},(c,d))$ and  \\
$h\in C_{rd}([a,b]_{\mathbb{T}},\mathbb{R})$ with
$$
\int_{a}^{b}|h(s)|\Delta s>0.
$$
If $f\in C((c,d),\mathbb{R}^{+}_{\mathcal{I}})$ is a concave function, then
\begin{equation*}
\begin{split}
&\frac{(IR)\int_{a}^{b}|h(s)|f(g(s))\Delta s}{\int_{a}^{b}|h(s)|\Delta s}
\supseteq f\Bigg(\frac{\int_{a}^{b}|h(s)|\underline{f(g)}(s)\Delta s}{
\int_{a}^{b}|h(s)|\Delta s}\Bigg).
\end{split}
\end{equation*}
\end{theorem}

\begin{theorem}
\label{thm4.8}
Let $g\in C_{rd}([a,b]_{\mathbb{T}},(c,d))$ and \\ 
$h\in C_{rd}([a,b]_{\mathbb{T}},\mathbb{R})$ with
$$
\int_{a}^{b}|h(s)|\Delta s>0.
$$
If $f\in C((c,d),\mathbb{R}^{+}_{\mathcal{I}})$ 
is an affine function, then
\begin{equation*}
\begin{split}
&\frac{(IR)\int_{a}^{b}|h(s)|f(g(s))\Delta s}{\int_{a}^{b}|h(s)|\Delta s}
= f\Bigg(\frac{\int_{a}^{b}|h(s)|
\underline{f(g)}(s)\Delta s}{\int_{a}^{b}|h(s)|\Delta s}\Bigg).
\end{split}
\end{equation*}
\end{theorem}

\begin{theorem}[Agarwal et al. \cite{A14}]
\label{thm4.9}
\\Let $f,g,h\in C_{rd}([a,b]_{\mathbb{T}},(0,\infty))$. 
If $\frac{1}{p}+\frac{1}{q}=1$, with $p>1$, then
\begin{equation*}
\begin{split}
&\int_{a}^{b}h(s)f(s)g(s)\Delta s\\
&\leq \Bigg(\int_{a}^{b}h(s)f^{p}(s)\Delta s\Bigg)^{\frac{1}{p}}\Bigg(
\int_{a}^{b}h(s)g^{q}(s)\Delta s\Bigg)^{\frac{1}{q}}.
\end{split}
\end{equation*}
\end{theorem}

Next we present a H\"{o}lder type inequality 
for interval-valued functions on time scales.

\begin{theorem}[H\"{o}lder's inequality]
\label{thm4.10}
\\Let $h\in C_{rd}([a,b]_{\mathbb{T}},(0,\infty))$, 
$f,g\in C_{rd}([a,b]_{\mathbb{T}},\mathbb{R}_{\mathcal{I}}^{+})$.\\
If $\frac{1}{p}+\frac{1}{q}=1$, with $p>1$, then
\begin{equation*}
\begin{split}
&\int_{a}^{b}h(s)f(s)g(s)\Delta s\\
&\leq \Bigg(\int_{a}^{b}h(s)f^{p}(s)\Delta s\Bigg)^{\frac{1}{p}}\Bigg(
\int_{a}^{b}h(s)g^{q}(s)\Delta s\Bigg)^{\frac{1}{q}}.
\end{split}
\end{equation*}
\end{theorem}

\begin{proof}
By hypothesis, we have
\begin{equation*}
\begin{split}
&\int_{a}^{b}h(s)f(s)g(s)\Delta s\\
&=\int_{a}^{b}h(s)\big[\underline{f}(s)\underline{g}(s),
\overline{f}(s)\overline{g}(s)\big]\Delta s\\
&=\Bigg[\int_{a}^{b}h(s)\underline{f}(s)\underline{g}(s)\Delta s,
\int_{a}^{b}h(s)\overline{f}(s)\overline{g}(s)\Delta s\Bigg]\\
&\leq \Bigg[\bigg(\int_{a}^{b}h(s)\underline{f}^{p}(s)\Delta s
\bigg)^{\frac{1}{p}}\bigg(\int_{a}^{b}h(s)\underline{g}^{q}(s)\Delta s\bigg)^{\frac{1}{q}},\\
&\ \ \ \   \bigg(\int_{a}^{b}h(s)\overline{f}^{p}(s)
\Delta s\bigg)^{\frac{1}{p}}\bigg(\int_{a}^{b}h(s)
\overline{g}^{q}(s)\Delta s\bigg)^{\frac{1}{q}}\Bigg]\\
&= \Bigg[\bigg(\int_{a}^{b}h(s)\underline{f}^{p}(s)\Delta s
\bigg)^{\frac{1}{p}},\bigg(\int_{a}^{b}h(s)\overline{f}^{p}(s)\Delta s\bigg)^{\frac{1}{p}}\Bigg]\\
&\ \ \ \  \cdot\Bigg[\bigg(\int_{a}^{b}h(s)\underline{g}^{q}(s)\Delta s\bigg)^{\frac{1}{q}},\bigg(\int_{a}^{b}h(s)\overline{g}^{q}(s)\Delta s\bigg)^{\frac{1}{q}}\Bigg]\\
&= \Bigg[\int_{a}^{b}h(s)\underline{f}^{p}(s)\Delta s,\int_{a}^{b}h(s)
\overline{f}^{p}(s)\Delta s\Bigg]^{\frac{1}{p}}\\
&\ \ \ \  \cdot\Bigg[\int_{a}^{b}h(s)\underline{g}^{q}(s)\Delta s,
\int_{a}^{b}h(s)\overline{g}^{q}(s)\Delta s\Bigg]^{\frac{1}{q}}\\
&=\Bigg(\int_{a}^{b}h(s)\Big[\underline{f}(s),\overline{f}(s)\Big]^{p}\Delta s\Bigg)^{\frac{1}{p}}\Bigg(\int_{a}^{b}h(s)\Big[\underline{g}(s),
\overline{g}(s)\Big]^{q}\Delta s\Bigg)^{\frac{1}{q}}\\
&=\Bigg(\int_{a}^{b}h(s)f^{p}(s)\Delta s\Bigg)^{\frac{1}{p}}\Bigg(
\int_{a}^{b}h(s)g^{q}(s)\Delta s\Bigg)^{\frac{1}{q}}.
\end{split}
\end{equation*}
This concludes the proof.
\end{proof}

For the particular case $p=q=2$ in Theorem~\ref{thm4.10}, 
we obtain the following Cauchy--Schwarz inequality.

\begin{theorem}[Cauchy--Schwarz inequality]
\label{thm4.11}
\\Let $h\in C_{rd}([a,b]_{\mathbb{T}},(0,\infty))$,
$f,g\in C_{rd}([a,b]_{\mathbb{T}},\mathbb{R}_{\mathcal{I}}^{+})$. Then,
\begin{equation*}
\begin{split}
&\int_{a}^{b}h(s)f(s)g(s)\Delta s\\
&\leq \sqrt{\Bigg(\int_{a}^{b}h(s)f^{2}(s)\Delta s\Bigg)\Bigg(
\int_{a}^{b}h(s)g^{2}(s)\Delta s\Bigg)}.
\end{split}
\end{equation*}
\end{theorem}

\begin{example}
\label{ex6}
Suppose that $[a,b]_{\mathbb{T}}=[0,\frac{\pi}{2}]$. Let
$h(s)=s$, $f(s)=[s,s+1]$, and
$g(s)=[\sin s,s]$ for $s\in[0,\frac{\pi}{2}]$.
Then
\begin{equation*}
\begin{split}
\int_{a}^{b}&h(s)f(s)g(s)\Delta s\\
&=\int_{0}^{\frac{\pi}{2}}\big[s^{2}\sin s,s^{3}+s^{2}\big]\Delta s\\
&=\bigg[\int_{0}^{\frac{\pi}{2}}s^{2}\sin s\Delta s,
\int_{0}^{\frac{\pi}{2}}(s^{3}+s^{2})\Delta s\bigg]\\
&=\bigg[\pi-2,\frac{\pi^{4}}{64}+\frac{\pi^{3}}{24}\bigg],
\end{split}
\end{equation*}
and
\begin{equation*}
\begin{split}
&\sqrt{\Bigg(\int_{a}^{b}h(s)f^{2}(s)\Delta s\Bigg)\Bigg(\int_{a}^{b}h(s)g^{2}(s)\Delta s\Bigg)}\\
&=\sqrt{\bigg(\int_{0}^{\frac{\pi}{2}}\big[s^{3},s^{3}+2s^{2}+s\big]\Delta s
\bigg)\bigg(\int_{0}^{\frac{\pi}{2}}\big[s\sin ^{2}s,s^{3}\big]\Delta s\bigg)}\\
&=\sqrt{\bigg[\int_{0}^{\frac{\pi}{2}}s^{3}ds,\int_{0}^{\frac{\pi}{2}}(s^{3}+2s^{2}
+s)ds\bigg]\cdot \bigg[\int_{0}^{\frac{\pi}{2}}s\sin ^{2}sds,\int_{0}^{\frac{\pi}{2}}s^{3}ds\bigg]}\\
&=\sqrt{\bigg[\frac{\pi^{4}}{64},\frac{\pi^{4}}{64}+\frac{\pi^{3}}{12}
+\frac{\pi^{2}}{8}\bigg]\cdot\bigg[\frac{\pi^{2}}{16}+\frac{1}{4},\frac{\pi^{4}}{64}\bigg]}\\
&=\sqrt{\bigg[\frac{\pi^{6}}{1024}+\frac{\pi^{4}}{256},\frac{\pi^{8}}{4096}
+\frac{\pi^{7}}{768}+\frac{\pi^{6}}{512}\bigg]}\\
&=\Bigg[\sqrt{\frac{\pi^{6}}{1024}+\frac{\pi^{4}}{256}},\sqrt{\frac{\pi^{8}}{4096}
+\frac{\pi^{7}}{768}+\frac{\pi^{6}}{512}}\Bigg].
\end{split}
\end{equation*}

Consequently, we obtain
\begin{equation*}
\begin{split}
&\bigg[\pi-2,\frac{\pi^{4}}{64}+\frac{\pi^{3}}{24}\bigg]\\
&\leq \Bigg[\sqrt{\frac{\pi^{6}}{1024}+\frac{\pi^{4}}{256}},
\sqrt{\frac{\pi^{8}}{4096}+\frac{\pi^{7}}{768}+\frac{\pi^{6}}{512}}\Bigg].
\end{split}
\end{equation*}
\end{example}

\begin{example}
\label{ex7}
Suppose that $[a,b]_{\mathbb{T}}=\{0,1,2,3\}$. Let
$h(s)=s$, $f(s)=[s,s+1]$, and
$g(s)=[\frac{s}{2},s]$ for $s\in\{0,1,2,3\}$.
Then
\begin{equation*}
\begin{split}
\int_{a}^{b}&h(s)f(s)g(s)\Delta s\\
&=\int_{0}^{3}\Big[\frac{s^{3}}{2},s^{3}+s^{2}\Big]\Delta s\\
&=\bigg[\int_{0}^{3}\frac{s^{3}}{2}\Delta s,\int_{0}^{3}s^{3}+s^{2}\Delta s\bigg]\\
&=\bigg[\frac{9}{2},14\bigg],
\end{split}
\end{equation*}
and
\begin{equation*}
\begin{split}
&\sqrt{\Bigg(\int_{a}^{b}h(s)f^{2}(s)\Delta s\Bigg)\Bigg(
\int_{a}^{b}h(s)g^{2}(s)\Delta s\Bigg)}\\
&=\sqrt{\bigg(\int_{0}^{3}\big[s^{3},s^{3}+2s^{2}+s\big]\Delta s\bigg)
\bigg(\int_{0}^{3}\Big[\frac{s^{3}}{4},s^{3}\Big]\Delta s\bigg)}\\
&=\sqrt{[9,22]\cdot\Big[\frac{9}{4},9\Big]}\\
&=\bigg[\frac{9}{2},3\sqrt{22}\bigg].
\end{split}
\end{equation*}
Consequently, we obtain
$$
\bigg[\frac{9}{2},14\bigg]\leq \bigg[\frac{9}{2},3\sqrt{22}\bigg].
$$
\end{example}

\begin{theorem}[Agarwal et al.\cite{A14}; Wong et al.\cite{W05}]
\label{thm4.12}
Let $f,g,h\in C_{rd}([a,b]_{\mathbb{T}},\mathbb{R})$ and  $p>1$. Then,
\begin{equation*}
\begin{split}
&\bigg(\int_{a}^{b}|h(s)||f(s)+g(s)|^{p}\Delta s\bigg)^{\frac{1}{p}}\\
&\leq \bigg(\int_{a}^{b}|h(s)||f(s)|^{p}\Delta s\bigg)^{\frac{1}{p}}
+\bigg(\int_{a}^{b}|h(s)||g(s)|^{p}\Delta s\bigg)^{\frac{1}{p}}.
\end{split}
\end{equation*}
\end{theorem}

By the same technique used in the proof of Theorem~4 
in \cite{R16}, we get a more general result.

\begin{theorem}[Minkowski's inequality]
\label{thm4.13}
\\Let $h\in C_{rd}([a,b]_{\mathbb{T}},\mathbb{R})$, 
$f,g\in C([a,b]_{\mathbb{T}},\mathbb{R}_{\mathcal{I}}^{+})$ 
and \\$p>1$. Then,
\begin{equation*}
\begin{split}
&\bigg(\int_{a}^{b}|h(s)|(f(s)+g(s))^{p}\Delta s\bigg)^{\frac{1}{p}}\\
&\leq \bigg(\int_{a}^{b}|h(s)|f^{p}(s)\Delta s\bigg)^{\frac{1}{p}}
+\bigg(\int_{a}^{b}|h(s)|g^{p}(s)\Delta s\bigg)^{\frac{1}{p}}.
\end{split}
\end{equation*}
\end{theorem}

\begin{proof}
By hypothesis, we have
\begin{equation*}
\begin{split}
&\bigg(\int_{a}^{b}|h(s)|(f(s)+g(s))^{p}\Delta s\bigg)^{\frac{1}{p}}\\
&=\bigg(\int_{a}^{b}|h(s)|\big[\underline{f}(s)+\underline{g}(s),
\overline{f}(s)+\overline{g}(s)\big]^{p}\Delta s\bigg)^{\frac{1}{p}}\\
&=\bigg(\int_{a}^{b}|h(s)|\big[(\underline{f}(s)
+\underline{g}(s))^{p},(\overline{f}(s)+\overline{g}(s))^{p}\big]\Delta s\bigg)^{\frac{1}{p}}\\
&=\Bigg[\bigg(\int_{a}^{b}|h(s)|(\underline{f}(s)+\underline{g}(s))^{p}\Delta s\bigg)^{\frac{1}{p}},\\
&\ \ \ \ \ \ \ \ \ \bigg(\int_{a}^{b}|h(s)|(\overline{f}(s)
+\overline{g}(s))^{p}\Delta s\bigg)^{\frac{1}{p}}\Bigg]\\
&\leq \Bigg[\bigg(\int_{a}^{b}|h(s)|\underline{f}^{p}(s)\Delta s\bigg)^{\frac{1}{p}}
+\bigg(\int_{a}^{b}|h(s)|\underline{g}^{p}(s)\Delta s\bigg)^{\frac{1}{p}},\\
&\ \ \ \bigg(\int_{a}^{b}|h(s)|\overline{f}^{p}(s)\Delta s\bigg)^{\frac{1}{p}}
+\bigg(\int_{a}^{b}|h(s)|\overline{g}^{p}(s)\Delta s\bigg)^{\frac{1}{p}}\Bigg]\\
&= \Bigg[\bigg(\int_{a}^{b}|h(s)|\underline{f}^{p}(s)\Delta s\bigg)^{\frac{1}{p}},
\bigg(\int_{a}^{b}|h(s)|\overline{f}^{p}(s)\Delta s\bigg)^{\frac{1}{p}}\Bigg]\\
&\ \  +\Bigg[\bigg(\int_{a}^{b}|h(s)|\underline{g}^{p}(s)\Delta s\bigg)^{\frac{1}{p}},
\bigg(\int_{a}^{b}|h(s)|\overline{g}^{p}(s)\Delta s\bigg)^{\frac{1}{p}}\Bigg]\\
&= \Bigg[\int_{a}^{b}|h(s)|\underline{f}^{p}(s)\Delta s,\int_{a}^{b}|h(s)|
\overline{f}^{p}(s)\Delta s\Bigg]^{\frac{1}{p}}\\
&\ \ \ \ \ \ +\Bigg[\int_{a}^{b}|h(s)|\underline{g}^{p}(s)\Delta s,
\int_{a}^{b}|h(s)|\overline{g}^{p}(s)\Delta s\Bigg]^{\frac{1}{p}}\\
&=\bigg(\int_{a}^{b}|h(s)|\big[\underline{f}(s),
\overline{f}(s)\big]^{p}\Delta s\bigg)^{\frac{1}{p}}\\
&\ \ \ \ \ +\bigg(\int_{a}^{b}|h(s)|\big[\underline{g}(s),
\overline{g}(s)\big]^{p}\Delta s\bigg)^{\frac{1}{p}}\\
&=\bigg(\int_{a}^{b}|h(s)|f^{p}(s)\Delta s\bigg)^{\frac{1}{p}}
+\bigg(\int_{a}^{b}|h(s)|g^{p}(s)\Delta s\bigg)^{\frac{1}{p}}.
\end{split}
\end{equation*}
The proof is complete.
\end{proof}

\begin{example}
\label{ex8}
Suppose that $[a,b]_{\mathbb{T}}=[0,1]\cup \{2\}$. 
Let\\ $h(s)=s$, $f(s)=[s,2s]$,
$g(s)=[s,e^{s}]$ and $p=2$.
Then,
\begin{equation*}
\begin{split}
&\bigg(\int_{a}^{b}|h(s)|(f(s)+g(s))^{p}\Delta s\bigg)^{\frac{1}{p}}\\
&=\sqrt{\int_{0}^{2}\big[4s^{3},se^{2s}+4s^{2}e^{s}+4s^{3}\big]\Delta s}\\
&=\sqrt{\bigg[\int_{0}^{2}4s^{3}\Delta s,\int_{0}^{2}se^{2s}+4s^{2}e^{s}+4s^{3}\Delta s\bigg]}\\
&=\Bigg[\sqrt{5},\frac{\sqrt{5e^{2}+32e-11}}{2}\Bigg],
\end{split}
\end{equation*}
and
\begin{equation*}
\begin{split}
&\bigg(\int_{a}^{b}|h(s)|f^{p}(s)\Delta s\bigg)^{\frac{1}{p}}
+\bigg(\int_{a}^{b}|h(s)|g(^{p}s)\Delta s\bigg)^{\frac{1}{p}}\\
&=\sqrt{\int_{0}^{2}\big[s^{3},4s^{3}\big]\Delta s}
+\sqrt{\int_{0}^{2}\big[s^{3},se^{2s}\big]\Delta s}\\
&=\bigg[\frac{\sqrt{5}}{2},\sqrt{5}\bigg]+\bigg[
\frac{\sqrt{5}}{2},\frac{\sqrt{5e^{2}+1}}{2}\bigg]\\
&=\bigg[\sqrt{5},\sqrt{5}+\frac{\sqrt{5e^{2}+1}}{2}\bigg].
\end{split}
\end{equation*}
Consequently, we obtain
$$
\Bigg[\sqrt{5},\frac{\sqrt{5e^{2}+32e-11}}{2}\Bigg]
\leq \bigg[\sqrt{5},\sqrt{5}+\frac{\sqrt{5e^{2}+1}}{2}\bigg].
$$
\end{example}

The next results follow directly from 
Theorems~\ref{thm4.10} and \ref{thm4.13}, respectively.

\begin{corollary}
\label{cor4.1}
Let $h\in C_{rd}([a,b]_{\mathbb{T}},(0,\infty))$, and\\
$f,g\in C_{rd}([a,b]_{\mathbb{T}},\mathbb{R}_{\mathcal{I}}^{-})$. 
If $\frac{1}{p}+\frac{1}{q}=1$, with $p>1$, then
\begin{equation*}
\begin{split}
&\int_{a}^{b}h(s)f(s)g(s)\Delta s\\
&\leq \Bigg(\int_{a}^{b}h(s)(-f)^{p}(s)\Delta s\Bigg)^{\frac{1}{p}}\Bigg(
\int_{a}^{b}h(s)(-g)^{q}(s)\Delta s\Bigg)^{\frac{1}{q}}.
\end{split}
\end{equation*}
\end{corollary}

\begin{corollary}
\label{cor4.2}
Let $h\in C_{rd}([a,b]_{\mathbb{T}},\mathbb{R})$, 
$f,g\in C([a,b]_{\mathbb{T}},\mathbb{R}_{\mathcal{I}}^{-})$ 
and $p\in2^{\mathbb{N}}$. Then,
\begin{multline*}
\bigg(\int_{a}^{b}|h(s)|(f(s)+g(s))^{p}\Delta s\bigg)^{\frac{1}{p}}\\
\leq \bigg(\int_{a}^{b}|h(s)|f^{p}(s)\Delta s\bigg)^{\frac{1}{p}}
+\bigg(\int_{a}^{b}|h(s)|g^{p}(s)\Delta s\bigg)^{\frac{1}{p}}.
\end{multline*}
\end{corollary}

% --------------------------------------------

\section{Conclusion}
\label{sec:5}

We investigated Darboux and Riemann interval delta integrals
for interval-valued functions on time scales. Inequalities
for interval-valued functions were proved. Our results generalize
previous inequalities presented by Costa 
\cite[Corollary 3.3, Theorem~3.4, Theorem~3.5]{C17} 
and Rom\'{a}n-Flores \cite[Theorem~4]{R16}.

% --------------------------------------------

\begin{acknowledgements}
This study was funded by Fundamental Research Funds
for the Central Universities (Grant Numbers 2017B19714 and 2017B07414).
Torres was supported by FCT and CIDMA,
project UID/MAT/04106/2013.

The authors are very grateful to two anonymous referees,
for several valuable and helpful comments, suggestions
and questions, which helped them to improve the paper
into present form.

\bigskip

\noindent{\bf Compliance with ethical standards}

\medskip

\noindent {\bf Conflicts of interest} 
The authors declare that they have no conflict of interest.

\medskip

\noindent {\bf Ethical approval} 
This article does not contain any studies with human
participants or animals performed by any of the authors.
\end{acknowledgements}

% --------------------------------------------

% --------------------------------------------


\begin{thebibliography}{xx}

\bibitem{A14}
R. Agarwal, D. O'Regan, S. Saker,
\textit{Dynamic inequalities on time scales}, Springer, Cham, (2014).

\bibitem{B13}
B. Bede,
\textit{Mathematics of Fuzzy Sets and Fuzzy Logic},
Studies in Fuzziness and Soft Computing, \textbf{295},
Springer, Heidelberg, (2013).

\bibitem{BCT1}
N. Benkhettou, A. M. C. Brito da Cruz, D. F. M. Torres,
\textit{A fractional calculus on arbitrary time scales:
fractional differentiation and fractional integration},
Signal Process., \textbf{107} (2015), 230--237.
{\tt arXiv:1405.2813}

\bibitem{BCT2}
N. Benkhettou, A. M. C. Brito da Cruz, D. F. M. Torres,
\textit{Nonsymmetric and symmetric fractional calculus on arbitrary nonempty closed sets},
Math. Methods Appl. Sci. \textbf{39} (2016), 261--279.
{\tt arXiv:1502.07277}

\bibitem{BP1}
M. Bohner, A. Peterson,
\textit{Dynamic equations on time scales: an introduction with applications},
Birkh\"{a}user, Boston, MA (2001).

\bibitem{BP2}
M. Bohner, A. Peterson,
\textit{Advances in dynamic equations on time scales},
Birkh\"{a}user, Boston, MA (2003).

\bibitem{B}
W. W. Breckner,
\textit{Continuity of generalized convex and generalized concave set-valued functions},
 Rev. Anal. Num\'{e}r. Th\'{e}or. Approx., \textbf{22} (1993), 39--51.

\bibitem{C81}
O. Caprani, K. Madsen, L. B. Rall,
\textit{Integration of interval functions},
SIAM J. Math. Anal., \textbf{12} (1981), 321--341.

\bibitem{C12}
Y. Chalco-Cano, A. Flores-Franuli\v{c}, H. Rom\'{a}n-Flores,
\textit{Ostrowski type inequalities for interval-valued 
functions using generalized Hukuhara derivative},
Comput. Appl. Math., \textbf{31} (2012), 457--472.

\bibitem{CL15}
Y. Chalco-Cano, W. A. Lodwick, W. Condori-Equice,
\textit{Ostrowski type inequalities and applications 
in numerical integration for interval-valued functions},
Soft Comput., \textbf{19} (2015), 3293--3300.

\bibitem{C13}
Y. Chalco-Cano, A. Rufi\'{a}n-Lizana, H. Rom\'{a}n-Flores, M. D. Jim\'{e}nez-Gamero,
\textit{Calculus for interval-valued functions 
using generalized Hukuhara derivative and applications},
Fuzzy Sets and Systems, \textbf{219} (2013), 49--67.

\bibitem{C15}
Y. Chalco-Cano, G. N. Silva, A. Rufi\'{a}n-Lizana,
\textit{On the Newton method for solving fuzzy optimization problems},
Fuzzy Sets and Systems, \textbf{272} (2015), 60--69.

\bibitem{C17}
T. M. Costa,
\textit{Jensen's inequality type integral for fuzzy-interval-valued functions},
Fuzzy Sets and Systems, \textbf{327} (2017), 31--47.

\bibitem{CB}
T. M. Costa, H. Bouwmeester, W. A. Lodwick, C. Lavor,
\textit{Calculating the possible conformations arising from uncertainty
in the molecular distance geometry problem using constraint interval analysis},
Inform. Sci., \textbf{415-416} (2017), 41--52.

\bibitem{CC}
T. M. Costa, Y. Chalco-Cano, W. A. Lodwick, G. N. Silva,
\textit{Generalized interval vector spaces and interval optimization},
Inform. Sci., \textbf{311} (2015), 74--85.

\bibitem{CF}
T. M. Costa, H. Rom\'{a}n-Flores,
\textit{Some integral inequalities for fuzzy-interval-valued functions},
Inform. Sci., \textbf{420} (2017), 110--125.

\bibitem{D08}
C. Dinu,
\textit{Convex functions on time scales},
An. Univ. Craiova Ser. Mat. Inform., \textbf{35} (2008), 87--96.

\bibitem{FB15}
O. S. Fard, T. A. Bidgoli,
\textit{Calculus of fuzzy functions on time scales (I)},
 Soft Comput., \textbf{19} (2015), 293--305.

\bibitem{FT16}
O. S. Fard, D. F. M. Torres, M. R. Zadeh,
\textit{A Hukuhara approach to the study of hybrid fuzzy systems on time scales},
Appl. Anal. Discrete Math. \textbf{10} (2016), 152--167.
{\tt arXiv:1603.03737}

\bibitem{FC}
A. Flores-Franuli\v{c}, Y. Chalco-Cano, H. Rom\'{a}n-Flores,
\textit{An Ostrowski type inequality for interval-valued functions},
IFSA World Congress and NAFIPS Annual Meeting IEEE, \textbf{35} (2013), 1459--1462.

\bibitem{G17}
N. A. Gasilov, \c{S}. E. Amrahov,
\textit{Solving a nonhomogeneous linear system of interval differential equations},
Soft Comput.,  \textbf{22} (2018), 3817--3828.

\bibitem{H4}
S. Hilger,
\textit{Ein Ma{\ss}kettenkalk\"{u}l mit Anwendung auf Zent-\\rumsmannigfaltigkeiten},
Ph.D. Thesis, Universit\"{a}t W\"{u}rzb-\\urg (1988).

\bibitem{J01}
L. Jaulin, M. Kieffer, O. Didrit, \'{E}. Walter,
\textit{Applied interval analysis},
Springer-Verlag London, Ltd., London, (2001).

\bibitem{L13}
V. Lupulescu,
\textit{Hukuhara differentiability of interval-valued functions
and interval differential equations on time scales},
Inform. Sci., \textbf{248} (2013), 50--67.

\bibitem{L15}
V. Lupulescu,
\textit{Fractional calculus for interval-valued functions},
Fuzzy Sets and Systems, \textbf{265} (2015), 63--85.

\bibitem{L17}
V. Lupulescu, N. V. Hoa,
\textit{Interval Abel integral equation},
 Soft Comput., \textbf{21} (2017), 2777--2784.

\bibitem{M66}
R. E. Moore,
\textit{Interval analysis},
Prentice-Hall, Inc., Englewood Cliffs, N.J., (1966).

\bibitem{M79}
R. E. Moore,
\textit{Methods and applications of interval analysis},
SIAM, Philadelphia, Pa., (1979).

\bibitem{M09}
R. E. Moore, R. B. Kearfott, M. J. Cloud,
\textit{Introduction to interval analysis},
SIAM, Philadelphia, PA., (2009).

\bibitem{O15}
R. Osuna-G\'{o}mez, Y. Chalco-Cano, B. Hern\'{a}ndez-Jim\'{e}nez, G. Ruiz-Garz\'{o}n,
\textit{Optimality conditions for generalized differentiable interval-valued functions},
Inform. Sci., \textbf{321} (2015), 136--146.

\bibitem{R82}
L. B. Rall,
\textit{Integration of interval functions. II. The finite case},
SIAM J. Math. Anal., \textbf{13} (1982), 690--697.

\bibitem{R16}
H. Rom\'{a}n-Flores, Y. Chalco-Cano, W. A. Lodwick,
\textit{Some integral inequalities for interval-valued functions},
Comput. Appl. Math., \textbf{37} (2018), 1306--1318.

\bibitem{R13}
H. Rom\'{a}n-Flores, Y. Chalco-Cano, G. N. Silva,
\textit{A note on Gronwall type inequality for interval-valued functions},
IFSA World Congress and NAFIPS Annual Meeting IEEE, \textbf{35} (2013), 1455--1458.

\bibitem{S09}
L. Stefanini, B. Bede,
\textit{Generalized Hukuhara differentiability of interval-valued
functions and interval differential equations},
Nonlinear Anal., \textbf{71} (2009), 1311--1328.

\bibitem{V16}
C. Vasavi, G. S. Kumar, M. S. N. Murty,
\textit{Generalized differentiability and integrability 
for fuzzy set-valued functions on time scales},
Soft Comput., \textbf{20} (2016), 1093--1104.

\bibitem{W06}
F.-H. Wong, C.-C. Yeh, W.-C. Lian,
\textit{An extension of Jensen's inequality on time scales},
Adv. Dyn. Syst. Appl., \textbf{1} (2006), 113--120.

\bibitem{W05}
F.-H. Wong, C.-C. Yeh, S.-L. Yu, C.-H. Hong,
\textit{Young's inequality and related results on time scales},
Appl. Math. Lett., \textbf{18} (2005), 983--988.

\bibitem{WG00}
C. X. Wu, Z. T. Gong,
\textit{On Henstock integrals of interval-valued functions and fuzzy-valued functions},
Fuzzy Sets and Systems, \textbf{115} (2000), 377--391.

\bibitem{YZ}
X. X. You, D. F. Zhao,
\textit{On convergence theorem for the McShane integral on time scales},
J. Chungcheong Math. Soc., \textbf{25} (2012), 393--400.

\bibitem{YZT}
X. X. You, D. F. Zhao, D. F. M. Torres,
\textit{On the Henstock-\\Kurzweil integral 
for Riesz-space-valued functions on time scales},
J. Nonlinear Sci. Appl., \textbf{10} (2017), 2487--2500.
{\tt arXiv:1704.06808}

\bibitem{ZL16}
D. F. Zhao, T. X. Li,
\textit{On conformable delta fractional calculus on time scales},
J. Math. Computer Sci., \textbf{16} (2016), 324--335.

\bibitem{Z17}
P. Z. Zhou, J. B. Du, Z. H. L\"{U},
\textit{Interval analysis based robust truss optimization with continuous
and discrete variables using mix-coded genetic algorithm},
Struct. Multidiscip. Optim., \textbf{56} (2017), 353--370.

\end{thebibliography}
\end{document}